\documentclass[11pt]{amsart}
\usepackage{amssymb, amsmath, amsfonts, amsthm, esint, graphicx, xcolor}
\usepackage[abbrev]{amsrefs}
\usepackage[margin=0.9in]{geometry}
\usepackage[colorlinks=true, linkcolor=blue, citecolor=blue]{hyperref}
\usepackage{tikz}
\usetikzlibrary{arrows.meta, positioning, calc, decorations.pathreplacing, shapes.geometric}

\DeclareMathOperator{\Ric}{Ric}
\DeclareMathOperator{\Hess}{Hess}
\DeclareMathOperator{\diam}{diam}
\DeclareMathOperator{\vol}{vol}

\DeclareMathOperator{\intt}{int}

\newtheorem{theorem}{Theorem}[section]
\newtheorem{definition}[theorem]{Definition}
\newtheorem{lemma}[theorem]{Lemma}
\newtheorem{proposition}[theorem]{Proposition}
\newtheorem{corollary}[theorem]{Corollary}
\newtheorem{remark}[theorem]{Remark}

\newtheorem{thmx}{Theorem}

\title[Spectral Gap Estimates on Nearly Einstein Manifolds ]{Spectral Gap Estimates on Nearly Einstein Manifolds with Integral Curvature Deviation}

\author{Shunzi Guo}
\address{School of Mathematics, Yunnan Normal University, Kunming 650500, P. R. China}
\email{guoshunzi@yeah.net}
\thanks{The author is supported by the National Natural Science Foundation of China (Grant No.~12261105).}

\subjclass[2020]{58C40, 35P15, 53C21, 53C25}
\keywords{Spectral gap, nearly Einstein manifolds, integral curvature bounds, Bakry--Qian estimate}

\date{}

\begin{document}

\begin{abstract}
Let $(M^n,g)$ be a compact Riemannian manifold of dimension $n\ge 2$,
possibly with nonempty convex $C^2$ boundary. We study the first nonzero
eigenvalue $\lambda_1(M)$ of the Laplace--Beltrami operator, subject to the
Neumann boundary condition when $\partial M\neq\emptyset$. Given a reference
scalar curvature $R_E>0$, set $K_E:=R_E/(n(n-1))$, and define the traceless
Ricci tensor $\mathring{\Ric}:=\Ric-\frac{R}{n}g$ and the Einstein deviation
\[
\Theta:=|\mathring{\Ric}|+\frac{|R-R_E|}{n},\qquad
\kappa_{\mathrm{Ein}}(p):=\left(\frac{1}{\vol(M)}\int_M\Theta^p\,dV\right)^{1/p},
\quad p>n/2.
\]
We say that $M$ is $\varepsilon$-nearly Einstein if
$\kappa_{\mathrm{Ein}}(p)<\varepsilon$. The main result states that for every
$\alpha\in(0,1)$ there exists
$\varepsilon_0=\varepsilon_0(n,p,\alpha,D,R_E)>0$ such that if
$\diam(M)\le D$ and $\kappa_{\mathrm{Ein}}(p)<\varepsilon_0$, then
\[
\lambda_1(M)\ge \alpha\,\lambda_1(n,K_E,D),
\]
where $\lambda_1(n,K,D)$ is the first nonzero eigenvalue of the Bakry--Qian
one-dimensional model problem with curvature parameter $K$, with the convention
$\lambda_1(n,K,D)=nK$ when $K>0$ and $D\ge\pi/\sqrt{K}$. The proof combines the
Bochner formula for the decomposition
$\Ric=(n-1)K_Eg+\theta g+\mathring{\Ric}$ with an auxiliary function $J$
satisfying
\[
\Delta J-\tau\frac{|\nabla J|^2}{J}-2J\Theta=-\sigma J
\]
and Neumann boundary conditions, together with Moser iteration and a maximum
principle. As $\kappa_{\mathrm{Ein}}(p)\to0$, the estimate recovers the
classical Bakry--Qian theorem.
\end{abstract}
\maketitle
\vspace{-0.4cm}

\section{Introduction}\label{sec:intro}

\subsection{Background and motivation}

Let $(M^n, g)$ be an $n$-dimensional compact Riemannian manifold, where $n \ge 2$ denotes the dimension, $g$ the Riemannian metric, and $\nabla$ the Levi--Civita connection associated with $g$. We assume that $M$ has possibly nonempty boundary $\partial M$, which is required to be convex and of class $C^2$ whenever it is nonempty; $\nu$ denotes the outward unit normal vector field along $\partial M$. We consider the eigenvalue problem of the Laplacian on $M$,
\begin{equation}\label{eq:eigen}
\Delta u + \lambda u = 0 \quad \text{in } M,
\end{equation}
imposing Neumann boundary conditions (vanishing of the normal derivative) in the case of nonempty boundary. Here $\Delta$ is the Laplace--Beltrami operator, $u$ is a smooth real-valued eigenfunction, and $\lambda$ is the corresponding eigenvalue. By the spectral theory of compact self-adjoint operators, the eigenvalues form a discrete sequence
\[
\lambda_0 = 0 < \lambda_1 \le \lambda_2 \le \cdots \to +\infty.
\]
The eigenvalue $\lambda_0 = 0$ corresponds to constant eigenfunctions, and $\lambda_1$ denotes the first nonzero eigenvalue, which is the quantity of interest in this paper. The study of $\lambda_1$ occupies a central place in geometric analysis, and the problem of obtaining quantitative lower bounds in terms of geometric data has attracted sustained attention over several decades.

The classical result in this direction is due to Lichnerowicz \cite{lichnerowicz1958geometrie}, who showed that for compact manifolds with a positive Ricci curvature lower bound $\Ric \ge (n-1)K > 0$, where $\Ric$ denotes the Ricci tensor and $K$ a real constant, the first nonzero eigenvalue satisfies $\lambda_1(M) \ge nK$, with equality if and only if $M$ is isometric to the round sphere of radius $1/\sqrt{K}$. In the complementary regime of nonnegative Ricci curvature, Zhong and Yang \cite{Zhong1984ONTE} proved that $\lambda_1(M) \ge \pi^2/D^2$, where $D := \diam(M)$ denotes the diameter of $M$. These two estimates reflect somewhat different geometric mechanisms: the Lichnerowicz bound is governed by the local curvature positivity, whereas the Zhong--Yang bound is governed by the global diameter constraint.

Figure~\ref{fig:setting} illustrates the geometric setting. The manifold $M$ is a compact Riemannian manifold with a nonempty convex boundary $\partial M$; the first nontrivial Neumann eigenfunction $u$ attains its maximum $u^* := \max_M u$ in the interior and its minimum $-1$ on the boundary. The eigenvalue $\lambda_1$ is determined by the interplay between the Ricci curvature lower bound and the diameter.

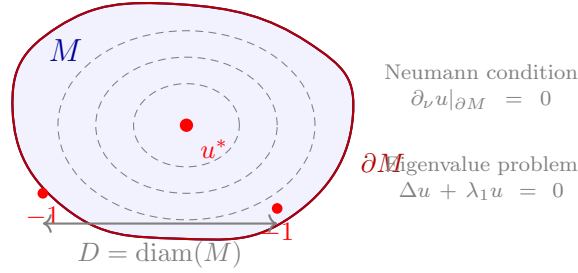
\begin{figure}[ht]
\centering
\begin{tikzpicture}[scale=1.0]
  \draw[thick, blue!60!black, fill=blue!5] plot[smooth cycle, tension=0.8] coordinates {(-2.2,0.2) (-1.6,1.3) (0,1.6) (1.7,1.1) (2.3,0.1) (1.5,-1.2) (0,-1.5) (-1.5,-1.1)};
  \node[blue!60!black, font=\large] at (-1.5,1.0) {$M$};
  \draw[thick, red!70!black] plot[smooth cycle, tension=0.8] coordinates {(-2.2,0.2) (-1.6,1.3) (0,1.6) (1.7,1.1) (2.3,0.1) (1.5,-1.2) (0,-1.5) (-1.5,-1.1)};
  \node[red!70!black, font=\small] at (2.7,-0.5) {$\partial M$};
  \draw[gray, densely dashed] (0.1,0) ellipse (0.7 and 0.55);
  \draw[gray, densely dashed] (0.1,0) ellipse (1.2 and 0.9);
  \draw[gray, densely dashed] (0.1,0) ellipse (1.7 and 1.25);
  \fill[red] (0.1,0) circle (2.5pt);
  \node[red, below right, font=\small] at (0.15,-0.05) {$u^*$};
  \fill[red] (-1.8,-0.9) circle (2.0pt);
  \fill[red] (1.3,-1.1) circle (2.0pt);
  \node[red, below, font=\small] at (-1.8,-0.95) {$-1$};
  \node[red, below, font=\small] at (1.3,-1.15) {$-1$};
  \draw[<->, thick, gray] (-1.8,-1.3) -- (1.3,-1.3);
  \node[gray, below, font=\small] at (-0.25,-1.4) {$D = \diam(M)$};
  \node[font=\scriptsize, gray, text width=4cm, align=center] at (4.0,0.5) {Neumann condition\\$\partial_\nu u|_{\partial M} = 0$};
  \node[font=\scriptsize, gray, text width=4cm, align=center] at (4.0,-0.7) {Eigenvalue problem\\$\Delta u + \lambda_1 u = 0$};
\end{tikzpicture}
\caption{The geometric setting. The manifold $M$ has a nonempty convex boundary $\partial M$; the first nontrivial Neumann eigenfunction $u$ attains its maximum $u^* := \max_M u$ in the interior and its minimum $-1$ on the boundary. The diameter $D := \diam(M)$ and the Ricci curvature bound together control the first nonzero eigenvalue $\lambda_1$.}
\label{fig:setting}
\end{figure}

Bakry and Qian \cite{bakry2000some} (see also Kr\"oger \cite{kroger1992spectral}) succeeded in unifying these two estimates within a single framework, by introducing a one-dimensional model problem whose first nonzero eigenvalue depends simultaneously on the dimension, the curvature lower bound, and the diameter. Their result may be stated as follows.

\begin{thmx}[Bakry--Qian \cite{bakry2000some}]\label{thm:bakryqian}
Let $M$ be an $n$-dimensional compact Riemannian manifold, possibly with nonempty convex and $C^2$ boundary, satisfying $\Ric \ge (n-1)K$ for some $K \in \mathbb{R}$ and $\diam(M) \le D$, where $D > 0$ is a prescribed upper bound for the diameter. Then
\[
\lambda_1(M) \ge \lambda_1(n, K, D),
\]
where $\lambda_1(n, K, D)$ is the first nonzero eigenvalue of the one-dimensional problem
\begin{equation}\label{eq:model}
\begin{cases}
w'' - T_{n,K} w' + \lambda w = 0 & \text{on } \left[-\frac{D}{2}, \frac{D}{2}\right], \\[4pt]
w'\left(\pm\frac{D}{2}\right) = 0,
\end{cases}
\end{equation}
with the weight function
\[
T_{n,K}(x) = \begin{cases}
(n-1)\sqrt{K}\tan(\sqrt{K}x), & K > 0, \\[2pt]
0, & K = 0, \\[2pt]
-(n-1)\sqrt{-K}\tanh(\sqrt{-K}x), & K < 0.
\end{cases}
\]
Here $w$ is the eigenfunction of the model problem, and $\lambda$ its eigenvalue.
\end{thmx}

We must be careful about the domain of definition of the model problem \eqref{eq:model}. When $K > 0$, the weight $T_{n,K}$ diverges at $x = \pm\pi/(2\sqrt{K})$, so that the model problem is well-defined only when $D < \pi/\sqrt{K}$. In the complementary regime $K > 0$ and $D \ge \pi/\sqrt{K}$, the diameter constraint is inactive: the sharp lower bound is $\lambda_1(M) \ge nK$, attained by the round sphere $S^n(1/\sqrt{K})$, and accordingly we adopt the convention
\begin{equation}\label{eq:lambda-segment}
\lambda_1(n, K, D) := nK \quad \text{when } K > 0 \text{ and } D \ge \frac{\pi}{\sqrt{K}}.
\end{equation}
This piecewise convention is standard in the Bakry--Qian theory \cite{bakry2000some} and ensures that the model eigenvalue is well-defined for all parameter values.

We emphasize that in Theorem \ref{thm:bakryqian} the quantity $K$ is a \emph{given} constant, prescribed by the pointwise Ricci curvature hypothesis $\Ric \ge (n-1)K$: it is the normalized lower bound of the Ricci tensor and, geometrically, it plays the role of a comparison curvature. It appears in two places in the statement: in the hypothesis, through the condition $\Ric \ge (n-1)K$, and in the conclusion, through the weight function $T_{n,K}$ of the one-dimensional model problem \eqref{eq:model}. The three cases $K>0$, $K=0$, and $K<0$ correspond, respectively, to the positively curved, flat, and negatively curved regimes, and the qualitative behaviour of the model eigenfunction $w$ differs accordingly: in the positive case, $T_{n,K}$ diverges at the endpoints of the symmetric interval and the eigenfunction is concentrated near the centre, whereas in the nonpositive case the weight remains bounded and the profile is flatter. In the present paper, we shall be concerned with the nearly Einstein setting, in which the natural candidate for $K$ is the curvature constant $K_E := R_E/(n(n-1))$ of the reference Einstein manifold; this will be made precise in \S\ref{sec:intro-nearly-Einstein}.

Figure~\ref{fig:model} illustrates the one-dimensional model problem and its first nontrivial Neumann eigenfunction. The weight function $T_{n,K}$ encodes the influence of the curvature lower bound: for $K>0$ it is unbounded at the endpoints of the symmetric interval, reflecting the focusing effect of positive curvature, whereas for $K \le 0$ it remains bounded, and the eigenfunction is correspondingly less concentrated near the centre.

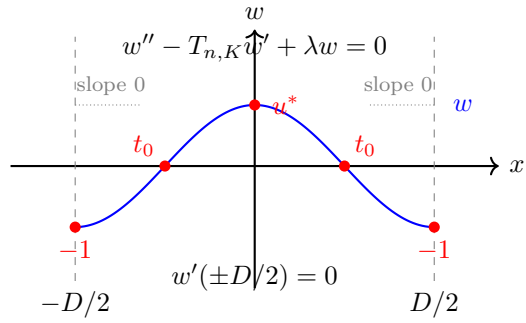
\begin{figure}[ht]
\centering
\begin{tikzpicture}[scale=0.95]
  \draw[->, thick] (-3.4,0) -- (3.4,0) node[right] {\small $x$};
  \draw[->, thick] (0,-1.7) -- (0,1.9) node[above] {\small $w$};
  \draw[dashed, gray] (-2.5,-1.6) -- (-2.5,1.8);
  \draw[dashed, gray] (2.5,-1.6) -- (2.5,1.8);
  \node[below, font=\small] at (-2.5,-1.6) {$-D/2$};
  \node[below, font=\small] at (2.5,-1.6) {$D/2$};
  \draw[thick, blue, domain=-2.5:2.5, samples=120] plot (\x, {0.85*cos(180*\x/2.5)});
  \node[blue, font=\small] at (2.9,0.85) {$w$};
  \draw[gray, densely dotted] (-2.5,0.85) -- (-1.6,0.85);
  \draw[gray, densely dotted] (2.5,0.85) -- (1.6,0.85);
  \node[gray, font=\scriptsize, above] at (-2.0,0.85) {slope $0$};
  \node[gray, font=\scriptsize, above] at (2.0,0.85) {slope $0$};
  \fill[red] (0,0.85) circle (2.2pt);
  \node[red, right, font=\small] at (0.1,0.85) {$u^*$};
  \fill[red] (-2.5,-0.85) circle (2.2pt);
  \fill[red] (2.5,-0.85) circle (2.2pt);
  \node[red, below, font=\small] at (-2.5,-0.9) {$-1$};
  \node[red, below, font=\small] at (2.5,-0.9) {$-1$};
  \fill[red] (-1.25,0) circle (2.2pt);
  \fill[red] (1.25,0) circle (2.2pt);
  \node[red, above left, font=\small] at (-1.25,0) {$t_0$};
  \node[red, above right, font=\small] at (1.25,0) {$t_0$};
  \node[font=\small] at (0,1.65) {$w'' - T_{n,K}w' + \lambda w = 0$};
  \node[font=\small] at (0,-1.55) {$w'(\pm D/2) = 0$};
\end{tikzpicture}
\caption{The one-dimensional Bakry--Qian model problem on the symmetric interval $[-D/2, D/2]$. The first nontrivial Neumann eigenfunction $w$ is shown here for the flat case $K=0$, where it is a cosine-like profile; it attains its extrema $\pm u^*$ at the endpoints and at the centre, and satisfies the Neumann boundary condition $w'(\pm D/2)=0$. The points where $w=0$ are marked by $t_0$. For $K>0$, the weight $T_{n,K}$ diverges at the endpoints, and the eigenfunction becomes more concentrated near the centre; for $K<0$, the profile becomes flatter than in the flat case.}
\label{fig:model}
\end{figure}

Andrews and Clutterbuck \cite{andrews2013sharp} later gave an independent proof of Theorem \ref{thm:bakryqian} using a two-point maximum principle, thereby clarifying the connection between the eigenvalue estimate and the modulus of continuity of solutions to the heat equation.

\begin{remark}\label{rem:interpolation}
The estimate of Theorem \ref{thm:bakryqian} interpolates between the Lichnerowicz and Zhong--Yang bounds. Indeed, a direct computation yields $\lambda_1(n,0,D) = \pi^2/D^2$ and $\lambda_1(n,K,\pi/\sqrt{K}) = nK$ for $K > 0$. More quantitatively, Shi and Zhang \cite{shi2007lower} obtained the interpolation inequality
\[
\lambda_1(n,K,D) \ge \max_{s \in (0,1)} 4(s-s^2)\frac{\pi^2}{D^2} + s(n-1)K.
\]
In the negative curvature regime, Yang \cite{yang1990_estimate_eigenvalue} derived the lower bound
\[
\lambda_1(M) \ge \frac{\pi^2}{D^2}\exp\left(-c_nD\sqrt{(n-1)|K|}\right),
\]
where $c_n = \max\{2, n-1\}$. These complementary estimates indicate that the behavior of $\lambda_1$ depends in a subtle way on the sign of the curvature bound.
\end{remark}

Figure~\ref{fig:interpolation} depicts qualitatively the interpolation between the two classical regimes. The dashed curve represents the Bakry--Qian eigenvalue $\lambda_1(n,K,D)$ for fixed $n$ and $K>0$; it approaches the Lichnerowicz bound $nK$ as the diameter becomes small, and the Zhong--Yang bound $\pi^2/D^2$ as the diameter becomes large.

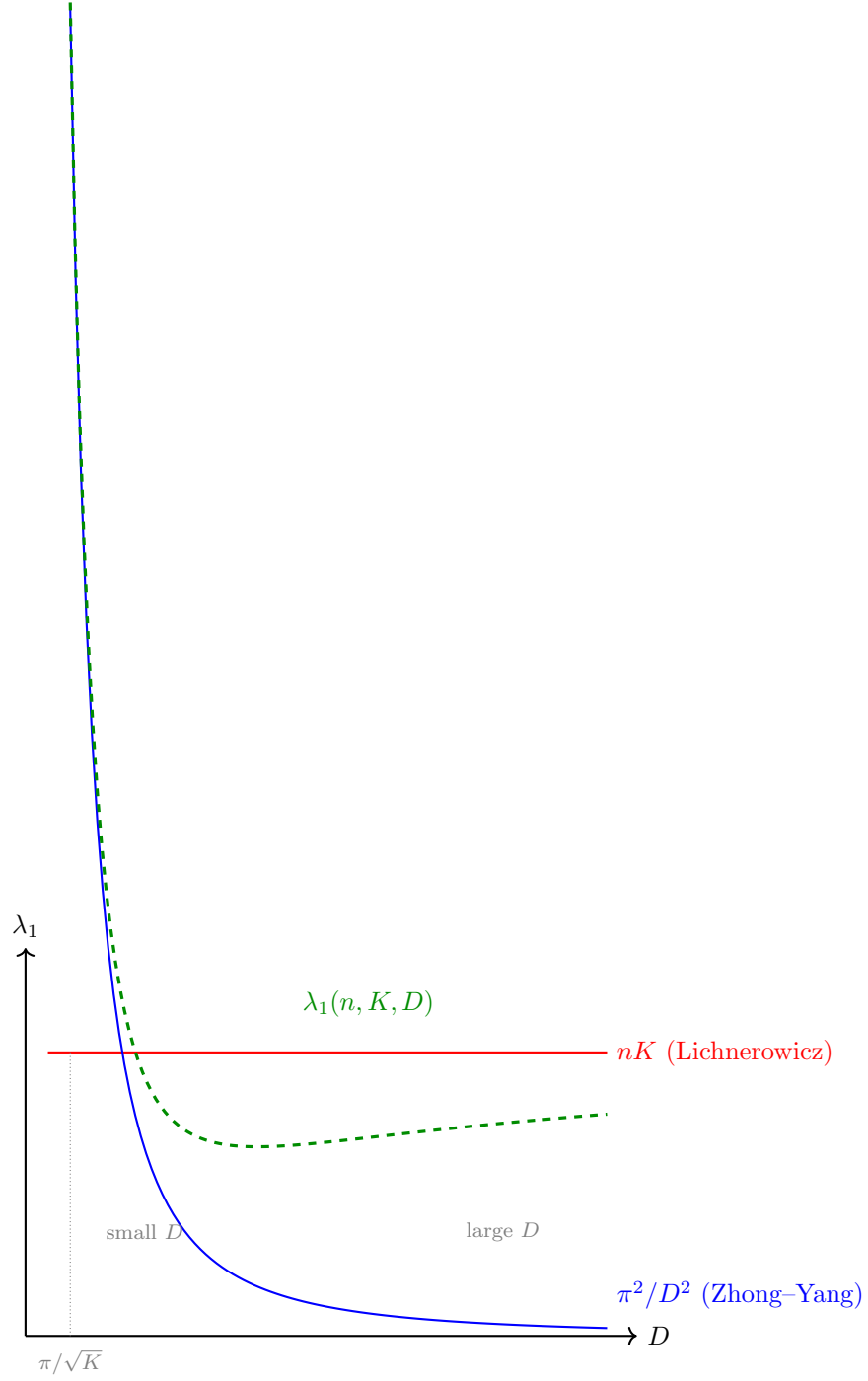
\begin{figure}[ht]
\centering
\begin{tikzpicture}[scale=1.0]
  \draw[->, thick] (0,0) -- (8.2,0) node[right] {\small $D$};
  \draw[->, thick] (0,0) -- (0,5.2) node[above] {\small $\lambda_1$};
  \draw[red, thick] (0.3,3.8) -- (7.8,3.8);
  \node[red, right, font=\small] at (7.8,3.8) {$nK$ (Lichnerowicz)};
  \draw[blue, thick, domain=0.6:7.8, samples=100] plot (\x, {1.6/(\x*\x/4)});
  \node[blue, right, font=\small] at (7.8,0.55) {$\pi^2/D^2$ (Zhong--Yang)};
  \draw[green!55!black, very thick, dashed, domain=0.6:7.8, samples=120] plot (\x, {3.8*exp(-2.2/\x)+1.6/(\x*\x/4)});
  \node[green!55!black, font=\small] at (4.6,4.45) {$\lambda_1(n,K,D)$};
  \draw[gray, densely dotted] (0.6,3.8) -- (0.6,0);
  \node[gray, below, font=\scriptsize] at (0.6,-0.05) {$\pi/\sqrt{K}$};
  \node[font=\scriptsize, gray] at (1.6,1.4) {small $D$};
  \node[font=\scriptsize, gray] at (6.4,1.4) {large $D$};
\end{tikzpicture}
\caption{Qualitative interpolation between the Lichnerowicz bound $nK$ (small diameter regime) and the Zhong--Yang bound $\pi^2/D^2$ (large diameter regime). The plot is qualitative only, drawn with $n=3$ and $K=1$ for illustration; the horizontal axis is the diameter $D$, the vertical axis the first nonzero eigenvalue $\lambda_1$. The dashed curve represents $\lambda_1(n,K,D)$.}
\label{fig:interpolation}
\end{figure}

\subsection{Integral curvature conditions and related work}

The hypotheses of Theorem \ref{thm:bakryqian} are pointwise in nature: the Ricci curvature is required to satisfy a lower bound at every point. In many geometric and analytic contexts, however, such a pointwise control is either unavailable or unnecessarily restrictive. This observation has motivated a substantial body of work aimed at replacing pointwise curvature bounds by integral ones, in which only the $L^p$-norm of the curvature deficit is controlled.

In this direction, Ramos Oliv\'e, Seto, and Tuerkoen \cite{RamosSetoTuerkoen} introduced the quantity
\[
\kappa(p, K) = \left(\frac{1}{\vol(M)}\int_M \rho_K^p \, dV\right)^{1/p},
\]
where $\vol(M)$ denotes the Riemannian volume of $M$, $dV$ its volume element, and, for $x \in M$, $\rho(x)$ denotes the smallest eigenvalue of the Ricci tensor $\Ric(x)$ and
\[
\rho_K(x) = \max\{-(\rho(x) - (n-1)K), 0\}
\]
measures the amount by which the smallest Ricci eigenvalue falls below $(n-1)K$ at the point $x$. Here, as in Theorem \ref{thm:bakryqian}, the constant $K$ is the given curvature lower bound. The exponent $p$ is required to satisfy $p > n/2$, which ensures that the relevant Sobolev inequalities hold. They proved that for any $\alpha \in (0,1)$ and $p > n/2$, there exists $\varepsilon_0 > 0$ such that whenever $\kappa(p, K) < \varepsilon_0$,
\[
\lambda_1(M) \ge \alpha \, \lambda_1(n, K, D).
\]
This confirmed a conjecture raised in \cite{ramos2020zhong}. The Lichnerowicz estimate with control on integral curvature had been established earlier by Aubry \cite{aubry2007finiteness}, who showed that for $K > 0$,
\[
\lambda_1(M) \ge nK(1 - C(n,p)\kappa(p,K)),
\]
where $C(n,p) > 0$ is a constant depending only on $n$ and $p$. The Zhong--Yang estimate, corresponding to the case $\Ric \ge 0$, was generalized to the integral curvature setting by Ramos Oliv\'e, Seto, Wei, and Zhang \cite{ramos2020zhong}. Rose and Wei \cite{rose-wei} subsequently extended the Zhong--Yang estimate to the closely related Kato-type condition.

Parallel to these developments, there has been considerable interest in understanding the structure of curvature tensors through their trace-free decompositions and the associated minimal norm principles. In a series of papers, Guo and Guan \cite{GuoGuan2021} and Guo and Li \cite{GuoLi2024, GuoLi2026} studied the minimal norm tensors of general third- and fourth-order covariant tensors, and applied these ideas to characterize Weyl and Cotton tensors, as well as to classify hypersurfaces whose covariant derivative of the Ricci tensor has vanishing minimal norm. These works provide a systematic framework for measuring the deviation of a given tensor from its canonical trace-free form, and the traceless Ricci tensor $\mathring{\Ric} := \Ric - \frac{R}{n}g$ that appears in the present paper may be viewed as a natural instance of this general principle. We emphasize, however, that the present paper does not rely on the minimal norm machinery; the traceless Ricci tensor enters only through its $L^p$-norm, which serves as a convenient quantitative measure of the deviation from the Einstein condition.

\subsection{The nearly Einstein condition}\label{sec:intro-nearly-Einstein}

The present paper is concerned with a different but closely related question. Rather than assuming that the Ricci curvature is bounded below by a constant multiple of the metric, we assume that the Ricci curvature is close to that of an Einstein manifold in an integral sense. This perspective is motivated by the observation that Einstein manifolds, characterized by the condition
\[
\Ric = \frac{R}{n}g,
\]
where $R$ denotes the scalar curvature, constitute a natural and geometrically significant class, and that many results in Riemannian geometry admit quantitative refinements in which the deviation from the Einstein condition plays the role of a small parameter.

To make this precise, we introduce the traceless Ricci tensor
\[
\mathring{\Ric} := \Ric - \frac{R}{n}g,
\]
which vanishes precisely when $M$ is Einstein. We then define the \emph{integral Einstein deviation} by
\[
\kappa_{\mathrm{Ein}}(p) := \left(\frac{1}{\vol(M)}\int_M \Theta^p \, dV\right)^{1/p},
\]
where
\[
\Theta(x) := \left|\mathring{\Ric}(x)\right| + \frac{|R(x) - R_E|}{n},
\]
$R_E > 0$ is a fixed reference scalar curvature, and $|\cdot|$ denotes the pointwise norm induced by $g$ on tensors. 
The quantity $\Theta$ combines two sources of deviation from the Einstein condition: the traceless part of the Ricci tensor, which measures the failure of the Ricci tensor to be a multiple of the metric, and the deviation of the scalar curvature from the reference value $R_E$. We note that $\kappa_{\mathrm{Ein}}(p) = 0$ if and only if $M$ is an Einstein manifold with scalar curvature $R_E$.

\begin{definition}[Nearly Einstein manifold]\label{def:nearly-Einstein}
Let $(M^n, g)$ be a compact Riemannian manifold, let $p > n/2$, and let $R_E > 0$ be a reference scalar curvature. For $\varepsilon > 0$, we say that $M$ is \emph{$\varepsilon$-nearly Einstein} with respect to $R_E$ and $p$ if
\[
\kappa_{\mathrm{Ein}}(p) < \varepsilon.
\]
\end{definition}

Equivalently, $M$ is $\varepsilon$-nearly Einstein if the $L^p$-norm of the combined deviation $\Theta$ is less than $\varepsilon$; that is, the Ricci curvature of $M$ is close, in an integral sense, to a multiple of the metric with constant $R_E/n$.

Throughout the paper, we work with a fixed reference scalar curvature $R_E > 0$ and set
\[
K_E := \frac{R_E}{n(n-1)}.
\]
This constant $K_E$ is the natural candidate for the curvature lower bound $K$ appearing in Theorem \ref{thm:bakryqian}: indeed, if $M$ is an Einstein manifold with scalar curvature $R_E$, then its Ricci tensor satisfies $\Ric = \frac{R_E}{n}g = (n-1)K_E g$, so that the pointwise Ricci bound $\Ric \ge (n-1)K_E$ holds with equality. In this sense, $K_E$ is the curvature constant associated with the reference Einstein manifold, and the nearly Einstein condition $\kappa_{\mathrm{Ein}}(p) < \varepsilon$ asserts that the Ricci curvature of $M$ is close, in an integral sense, to a multiple of the metric with constant $K_E$. The reader should keep in mind that throughout the paper, $K$ always denotes the curvature lower bound of the Bakry--Qian theorem, while $K_E$ denotes its natural Einstein reference value; the two coincide in the limiting case $\kappa_{\mathrm{Ein}}(p) = 0$.

Figure~\ref{fig:decomposition} illustrates the decomposition of the Ricci tensor into its Einstein part, its trace part, and its traceless part. The smallness of the $L^p$-norm of the last two components is precisely the content of the nearly Einstein condition.

\begin{figure}[ht]
\centering
\begin{tikzpicture}[scale=1.0]
  \node[font=\small] at (0,2.3) {$\Ric$};
  \node[font=\small] at (0.8,2.3) {$=$};
  \draw[fill=blue!25, rounded corners, draw=blue!60] (1.2,1.75) rectangle (3.4,2.85);
  \node[font=\small] at (2.3,2.3) {$(n-1)K_E\,g$};
  \node[font=\scriptsize, blue!70!black] at (2.3,1.5) {Einstein part};
  \node[font=\small] at (3.65,2.3) {$+$};
  \draw[fill=green!25, rounded corners, draw=green!60!black] (3.9,1.75) rectangle (5.4,2.85);
  \node[font=\small] at (4.65,2.3) {$\theta\,g$};
  \node[font=\scriptsize, green!55!black] at (4.65,1.5) {trace};
  \node[font=\small] at (5.65,2.3) {$+$};
  \draw[fill=red!25, rounded corners, draw=red!60] (5.9,1.75) rectangle (7.6,2.85);
  \node[font=\small] at (6.75,2.3) {$\mathring{\Ric}$};
  \node[font=\scriptsize, red!70!black] at (6.75,1.5) {traceless};
  \draw[decorate, decoration={brace, amplitude=5pt}, thick, gray] (3.9,1.05) -- (7.6,1.05);
  \node[font=\scriptsize, gray] at (5.75,0.75) {small in $L^p$ when $\kappa_{\mathrm{Ein}}(p)$ is small};
  \draw[->, thick, gray] (5.75,0.4) -- (5.75,-0.15);
  \node[font=\scriptsize, gray, text width=6cm, align=center] at (5.75,-0.45) {enters the Bochner formula as a controlled perturbation};
\end{tikzpicture}
\caption{Decomposition of the Ricci tensor into its Einstein part $(n-1)K_E\,g$, its trace part $\theta\,g$, and its traceless part $\mathring{\Ric}$, where $\theta := (R - R_E)/n$ and $K_E = R_E/(n(n-1))$. The nearly Einstein condition requires the last two components to be small in the $L^p$-norm, measured by the quantity $\kappa_{\mathrm{Ein}}(p)$. This decomposition is used in the Bochner estimate \eqref{eq:bochner}.}
\label{fig:decomposition}
\end{figure}

Figure~\ref{fig:sponge} visualizes the role of the auxiliary function $J$, which is constructed so as to absorb the Einstein deviation $\Theta$. The equation satisfied by $J$ is a nonlinear elliptic equation; its solution is close to the constant function $1$ when $\kappa_{\mathrm{Ein}}(p)$ is small, and this closeness is what allows the maximum principle argument to proceed.

\begin{figure}[ht]
\centering
\begin{tikzpicture}[scale=1.0]
  \begin{scope}[shift={(-3.6,0)}]
    \draw[thick, blue!60!black, fill=blue!5] plot[smooth cycle, tension=0.8] coordinates {(-1.5,0.2) (-1.1,1.0) (0,1.2) (1.1,0.9) (1.5,0.1) (1.0,-0.9) (0,-1.1) (-1.0,-0.8)};
    \foreach \r in {0.3,0.6,0.9} {
      \draw[gray!40, densely dotted] (0,0) circle (\r);
    }
    \fill[red] (0,0) circle (2.5pt);
    \node[red, right, font=\small] at (0.05,0) {$J \approx 1$};
    \node[font=\small] at (0,-1.5) {$J$ close to $1$};
    \node[font=\small] at (0,-1.85) {when $\kappa_{\mathrm{Ein}}(p)$ small};
  \end{scope}
  \draw[->, thick, gray] (-1.7,0) -- (-0.5,0);
  \node[font=\scriptsize, gray, above] at (-1.1,0.05) {Moser};
  \begin{scope}[shift={(2.0,0)}]
    \node[font=\small, text width=4.5cm, align=center] at (0,0.4) {$\Delta J - \tau\,\dfrac{|\nabla J|^2}{J} - 2J\Theta = -\sigma J$};
    \node[font=\scriptsize, gray, text width=5cm, align=center] at (0,-0.6) {with Neumann boundary condition\\$\partial_\nu J|_{\partial M} = 0$};
    \node[font=\scriptsize, gray, text width=5cm, align=center] at (0,-1.5) {$\sigma \le 4\varepsilon$, \quad $|J-1| \le \delta$};
  \end{scope}
\end{tikzpicture}
\caption{The auxiliary function $J$ (the ``sponge''), a positive solution of the nonlinear elliptic equation \eqref{eq:J} with Neumann boundary conditions; here $\tau > 1$ and $\sigma \ge 0$ are parameters. When the integral Einstein deviation $\kappa_{\mathrm{Ein}}(p)$ is small, $J$ is close to the constant function $1$, and it absorbs the Einstein deviation $\Theta$ in the Bochner estimate.}
\label{fig:sponge}
\end{figure}

The main result of this paper is the following.

\begin{theorem}[Main theorem]\label{thm:main}
Let \(M\) be an \(n\)-dimensional compact Riemannian manifold, possibly with
nonempty convex and \(C^2\) boundary, with \(n \ge 2\) and \(\diam(M) \le D\)
for some \(D>0\). Let \(p>n/2\), let \(R_E>0\) be a reference scalar
curvature, and set
\[
K_E := \frac{R_E}{n(n-1)}.
\]
Then for any \(\alpha\in(0,1)\), there exists
\(\varepsilon_0=\varepsilon_0(n,p,\alpha,D,R_E)>0\) such that whenever \(M\) is
\(\varepsilon_0\)-nearly Einstein with respect to \(R_E\) and \(p\)
(equivalently, whenever \(\kappa_{\mathrm{Ein}}(p)<\varepsilon_0\)), one has
\[
\lambda_1(M)\ge \alpha\,\lambda_1(n,K_E,D).
\]
Here, for a real parameter \(K\), \(\lambda_1(n,K,D)\) denotes the first
nonzero eigenvalue of the Bakry--Qian one-dimensional model problem
\eqref{eq:model} with curvature parameter \(K\) whenever the model problem is
well-defined, and we use the convention
\[
\lambda_1(n,K,D):=nK
\]
in the inactive regime \(K>0\), \(D\ge \pi/\sqrt{K}\), as in
\eqref{eq:lambda-segment}.
\end{theorem}
\begin{remark}\label{rem:RE-choice}
The reference scalar curvature $R_E$ should be chosen so that it is comparable to the average scalar curvature of $M$; otherwise, the quantity $\kappa_{\mathrm{Ein}}(p)$ may fail to be small for any positive $\varepsilon_0$. For instance, if $M$ is known to be close to an Einstein manifold with positive scalar curvature, one may take $R_E$ to be the average scalar curvature. The dependence of $\varepsilon_0$ on $R_E$ enters only through the choice of $\overline{K}$ in Theorem \ref{thm:gradient} and through the Aubry estimate \cite{aubry2007finiteness}.
\end{remark}

\begin{remark}\label{rem:sharp}
The estimate of Theorem \ref{thm:main} is consistent with the classical Bakry--Qian theorem in the following sense. As $\kappa_{\mathrm{Ein}}(p) \to 0$, we have $R \to R_E$ and $\mathring{\Ric} \to 0$, so that $M$ approaches an Einstein manifold with scalar curvature $R_E$. Setting $K = K_E = R_E/(n(n-1))$, the conclusion of Theorem \ref{thm:main} reduces, in the limit $\alpha \to 1$, to the conclusion of Theorem \ref{thm:bakryqian}. We do not claim that $\alpha = 1$ holds when $\kappa_{\mathrm{Ein}}(p) > 0$; the parameter $\alpha$ reflects the loss incurred by replacing the pointwise Ricci curvature bound with an integral Einstein deviation condition.
\end{remark}

\begin{remark}\label{rem:comparison-RST}
Compared with \cite{RamosSetoTuerkoen}, the key difference is that the reference
curvature is not a fixed constant $K$ but the curvature
$K_E = R_E/(n(n-1))$ of an Einstein manifold. Consequently, the deviation
$\Theta$ simultaneously controls the traceless Ricci tensor and the scalar
curvature deviation, and the auxiliary function equation must be constructed
to absorb this combined deviation. This leads to the new term $\sigma$ in the
gradient comparison and to the Aubry-type estimate
\cite{aubry2007finiteness} entering the choice of $\overline{K}$.
\end{remark}

\subsection{Strategy of the proof}

The proof of Theorem \ref{thm:main} adapts the strategy of Ramos Oliv\'e--Seto--Tuerkoen \cite{RamosSetoTuerkoen} to the setting of a varying reference curvature. Two modifications are required.

First, the auxiliary function $J$ (referred to in \cite{RamosSetoTuerkoen} as the ``sponge'') is constructed to absorb the Einstein deviation $\Theta$ rather than the Ricci curvature deficit $\rho_K$. The equation satisfied by $J$ takes the form
\[
\Delta J - \tau \frac{|\nabla J|^2}{J} - 2J\Theta = -\sigma J,
\]
with Neumann boundary conditions; here $\tau > 1$ and $\sigma \ge 0$ are parameters. The key property of $J$ is that it is close to $1$ in the $L^\infty$-norm when $\kappa_{\mathrm{Ein}}(p)$ is small; this is established by a Moser iteration argument, which requires a global Neumann Sobolev inequality and a Neumann Poincar\'e inequality under integral Einstein deviation bounds.

Second, the Bochner formula is applied to the decomposition
\[
\Ric = (n-1)K_E g + \theta g + \mathring{\Ric},
\]
where $K_E = R_E/(n(n-1))$ and $\theta = (R - R_E)/n$. This decomposition separates the Einstein part of the Ricci tensor from its traceless part and from the scalar curvature deviation, and thereby makes it possible to exploit the smallness of $\Theta$ in the Bochner estimate. We emphasize that the same quantity $\Theta$ appears both in the auxiliary function equation and in the Bochner estimate, which is essential for the maximum principle argument in Section \ref{sec:gradient}.

The logical structure of the proof may be described as follows. The gradient comparison (Section \ref{sec:gradient}) is the technical core; it relies on the auxiliary function $J$ and on the Bochner formula applied to the Ricci decomposition, and yields the pointwise estimate $J|\nabla u|^2 \le (w')^2 \circ (w^{-1}(u))$. The maximum comparison and maximum matching (Section \ref{sec:max}) then translate this gradient estimate into a statement about the maximum of the eigenfunction, by means of a reverse Poincar\'e argument and a continuity argument in the starting point of the one-dimensional initial value problem. Finally, the diameter comparison (Section \ref{sec:proof}) converts the maximum estimate into the eigenvalue estimate, using the domain monotonicity of the first Neumann eigenvalue of the model operator.

The remainder of the paper is organized as follows. In Section \ref{sec:prelim}, we fix notation and collect several preliminary results from the literature. In Section \ref{sec:gradient}, we prove the gradient comparison theorem, which is the technical core of the paper. In Section \ref{sec:max}, we establish the maximum comparison and the maximum matching theorem. In Section \ref{sec:proof}, we complete the proof of Theorem \ref{thm:main} by means of a diameter comparison. In Appendix \ref{app:sobolev}, we prove the global Neumann Sobolev and Poincar\'e inequalities under integral Einstein deviation bounds.

\section{Notation and Preliminaries}\label{sec:prelim}

\subsection{Basic notation}

Let $M$ be an $n$-dimensional compact Riemannian manifold with metric $g$, possibly with nonempty convex $C^2$ boundary $\partial M$, and let $n \ge 2$ denote its dimension. We denote by $u$ the eigenfunction for the first nonzero eigenvalue $\lambda_1$ of the Neumann problem \eqref{eq:eigen}, normalized so that
\[
-1 = \min_M u < 0 < \max_M u := u^* \le 1.
\]
Here $u^*$ denotes the maximum of $u$ on $M$. We denote by $\Hess u$ the Hessian of $u$, that is, the symmetric $(0,2)$-tensor $\Hess u := \nabla \nabla u$, and set
\[
A_u := \Hess u\left(\frac{\nabla u}{|\nabla u|}, \frac{\nabla u}{|\nabla u|}\right),
\]
which is the second derivative of $u$ in the direction of its gradient, evaluated on the unit vector $\nabla u / |\nabla u|$. This quantity will appear repeatedly in the maximum principle computations of Section \ref{sec:gradient}.

Let $R$ be the scalar curvature of $(M,g)$, $\Ric$ the Ricci tensor, and
\[
\mathring{\Ric} := \Ric - \frac{R}{n}g
\]
the traceless Ricci tensor, which vanishes precisely when $M$ is Einstein. We define, for $x \in M$,
\[
\eta(x) := |\mathring{\Ric}(x)|, \qquad \theta(x) := \frac{R(x) - R_E}{n}, \qquad \Theta(x) := \eta(x) + |\theta(x)|,
\]
where $|\cdot|$ denotes the norm induced by $g$. For $p > n/2$, the integral Einstein deviation is
\[
\kappa_{\mathrm{Ein}}(p) := \left(\frac{1}{\vol(M)}\int_M \Theta^p \, dV\right)^{1/p}.
\]
Throughout the paper, we work with a fixed reference scalar curvature $R_E > 0$ and set $K_E := R_E/(n(n-1))$. As explained in \S\ref{sec:intro-nearly-Einstein}, this constant $K_E$ is the natural candidate for the curvature lower bound $K$ of the Bakry--Qian theorem \ref{thm:bakryqian}. The choice of $R_E$ is discussed in Remark \ref{rem:RE-choice}.

\subsection{The one-dimensional model}

We recall the one-dimensional model introduced by Bakry and Qian. For $\overline{n} > 1$, $\overline{K} \in \mathbb{R}$, and $\overline{\lambda} > 0$, let $T$ be a solution to the Riccati equation
\begin{equation}\label{eq:riccati}
T' = \frac{T^2}{\overline{n}-1} + (\overline{n}-1)\overline{K}.
\end{equation}
We consider the initial value problem
\begin{equation}\label{eq:ivp}
\begin{cases}
w'' - T w' + \overline{\lambda} w = 0, \\[2pt]
w(a) = -1, \quad w'(a) = 0,
\end{cases}
\end{equation}
where $a \in \mathbb{R}$ is a starting point, $w$ is the unknown function, and $\overline{\lambda}$ is a positive parameter. We denote by $d(a, T, \overline{\lambda}) > 0$ the smallest positive number $d > a$ such that $w'(a+d) = 0$, and set $b := a + d(a, T, \overline{\lambda})$. The existence and uniqueness of $w$ follow from standard ODE techniques. The parameters $\overline{n}$, $\overline{K}$, and $\overline{\lambda}$ are perturbations of $n$, $K_E$, and $\lambda_1$, respectively, and will be specified in the proof of Theorem \ref{thm:gradient}.

We recall the following comparison result of Bakry and Qian \cite{bakry2000some} (Section 7, Theorem 13).

\begin{proposition}\label{prop:central}
For any $\overline{n} \ge 1$ and $\overline{K} \in \mathbb{R}$, let $T$ be any solution to the Riccati equation \eqref{eq:riccati}. Then
\[
\lambda_1(T, [a,b]) \ge \lambda_1\left(T_{\overline{n},\overline{K}}, \left[-\frac{d}{2}, \frac{d}{2}\right]\right),
\]
where $d = b - a$, $\lambda_1(T, [a,b])$ denotes the first Neumann eigenvalue of the operator $\mathcal{L}_T := \frac{d^2}{dt^2} - T\frac{d}{dt}$ on $[a,b]$, and $T_{\overline{n},\overline{K}}$ is the model solution given by
\[
T_{\overline{n},\overline{K}}(x) = \begin{cases}
(\overline{n}-1)\sqrt{\overline{K}}\tan(\sqrt{\overline{K}}x), & \overline{K} > 0, \\[2pt]
0, & \overline{K} = 0, \\[2pt]
-(\overline{n}-1)\sqrt{-\overline{K}}\tanh(\sqrt{-\overline{K}}x), & \overline{K} < 0.
\end{cases}
\]
\end{proposition}

We note that if $d = d(a, T, \overline{\lambda}) < \infty$, the number $\overline{\lambda} > 0$ is a Neumann eigenvalue of the operator $\mathcal{L}_T$ on $[a,b]$. While in general the Neumann eigenvalues do not satisfy the domain monotonicity property, the first Neumann eigenvalues of the operators $\mathcal{L}_T$ do satisfy it. That is, for intervals $I_1, I_2$,
\begin{equation}\label{eq:domain-monotonicity}
I_1 \subset I_2 \text{ implies } \lambda_1(T, I_1) \ge \lambda_1(T, I_2),
\end{equation}
where $\lambda_1(T, I_i)$ denotes the first Neumann eigenvalue of the operator $\mathcal{L}_T$ on the interval $I_i$. This monotonicity will be used in Section \ref{sec:proof} to compare the eigenvalue of the manifold with that of the one-dimensional model.

We denote by $d_{\overline{n},\overline{K},\overline{\lambda}}$ the length of the symmetric interval of the operator $\mathcal{L}_{T_{\overline{n},\overline{K}}}$ with eigenvalue $\overline{\lambda}$, that is,
\[
\overline{\lambda} = \lambda_1\left(T_{\overline{n},\overline{K}}, \left[-\frac{d_{\overline{n},\overline{K},\overline{\lambda}}}{2}, \frac{d_{\overline{n},\overline{K},\overline{\lambda}}}{2}\right]\right).
\]
Note that in the case $\overline{K} > 0$, we assume that $\overline{\lambda} > \overline{n}\overline{K}$, as otherwise such a $d$ might not exist (or $d = \pi/\sqrt{\overline{K}}$ for $\lambda = \overline{n}\overline{K}$). This restriction reflects the fact that the first nonzero Neumann eigenvalue of the model problem on a symmetric interval is bounded below by $\overline{n}\overline{K}$ when $\overline{K} > 0$.

\subsection{The auxiliary function}

We introduce the auxiliary function $J$, which plays the role of the ``sponge'' absorbing the Einstein deviation $\Theta$. The construction is motivated by the observation that, in the pointwise curvature case, the Ricci curvature deficit $\rho_K$ enters the Bochner formula as a negative term, and the auxiliary function serves to convert this negative term into a term that can be absorbed into the eigenvalue.

\begin{definition}\label{def:J}
For $K_E \in \mathbb{R}$, $\tau > 1$, and $\sigma \ge 0$, let $J > 0$ be a solution to
\begin{equation}\label{eq:J}
\Delta J - \tau \frac{|\nabla J|^2}{J} - 2J\Theta = -\sigma J \quad \text{in } M,
\end{equation}
with Neumann boundary conditions in the case $\partial M \neq \emptyset$. Here $\Theta$ is the same quantity that appears in the Bochner estimate \eqref{eq:bochner} below, $\Delta$ is the Laplace--Beltrami operator, and $|\nabla J|$ is the pointwise norm of the gradient of $J$.
\end{definition}

The following lemma establishes the existence of $J$ and shows that $J$ is close to $1$ in the $L^\infty$-norm when $\kappa_{\mathrm{Ein}}(p)$ is small. The proof relies on a Moser iteration argument, which in turn requires a global Neumann Sobolev inequality and a Neumann Poincar\'e inequality under integral Einstein deviation bounds; these are established in Appendix \ref{app:sobolev}.

\begin{lemma}\label{lem:J}
Let $(M,g)$ be a compact manifold, possibly with nonempty $C^2$ convex boundary, of dimension $n$ and diameter $D > 0$, and let $p > n/2$. For any $\delta > 0$ and $\tau > 1$, there exists $\varepsilon = \varepsilon(n, p, D, \tau) > 0$ such that if $\kappa_{\mathrm{Ein}}(p) \le \varepsilon$, then there exist $\sigma$ and a positive function $J$ solving \eqref{eq:J} such that
\[
0 \le \sigma \le 4\varepsilon, \qquad |J - 1| \le \delta.
\]
\end{lemma}

\begin{proof}
We make the transformation $J = W^{-1/(\tau-1)}$, where $W > 0$ is an auxiliary positive function. A direct computation gives
\[
\nabla J = -\frac{1}{\tau-1} W^{-\frac{1}{\tau-1}-1}\nabla W,
\]
and
\[
\Delta J - \tau \frac{|\nabla J|^2}{J} = -\frac{1}{\tau-1} W^{-\frac{1}{\tau-1}-1} \Delta W.
\]
Hence equation \eqref{eq:J} is equivalent to
\[
\Delta W + 2(\tau-1)\Theta W = (\tau-1)\sigma W,
\]
together with Neumann boundary conditions. Set $V := 2(\tau-1)\Theta$ and $\widetilde{\sigma} := (\tau-1)\sigma$. Since $\kappa_{\mathrm{Ein}}(p) \le \varepsilon$, we have $\|V\|_{L^p} \le 2(\tau-1)\varepsilon$. By the global Neumann Sobolev inequality and Neumann Poincar\'e inequality under integral Einstein deviation bounds (see Appendix \ref{app:sobolev}, Proposition \ref{prop:sobolev} and Corollary \ref{cor:poincare}), the Moser iteration argument of Ramos Oliv\'e--Seto--Wei--Zhang \cite{ramos2020zhong} (Proposition 2) applies verbatim with $\rho_K$ replaced by $\Theta$. This yields the existence of $\sigma \ge 0$ and a positive solution $W$ with $\|W - 1\|_{L^\infty} \le \delta$. Consequently, $|J - 1| \le \delta$ for $\delta$ sufficiently small.
\end{proof}

\begin{remark}\label{rem:J-wellposed}
The equation for $W$ is a standard Schr\"odinger-type equation with bounded potential $V \ge 0$, and the parameter $\widetilde{\sigma}$ is a principal eigenvalue of the associated operator. By the Krein--Rutman theorem, the corresponding eigenfunction can be chosen positive on $M$, which ensures the well-posedness of Definition \ref{def:J} and the positivity of $J$.
\end{remark}

\begin{remark}
The closed and $K = 0$ case was proved in \cite{ramos2020zhong}. The same argument extends to the present setting, with $K_E \in \mathbb{R}$ and possibly nonempty convex boundary. For related Sobolev inequalities under integral curvature assumptions in the boundaryless, $K = 0$, or local setting, see \cite{gallot1988isoperimetric, petersen1998integral, dai2018local, WangWei, RamosSeto}. We emphasize that the convexity of the boundary is used only to ensure that minimizing geodesics joining interior points remain in the interior, which is needed in the proof of the Sobolev inequality; see Lemma \ref{lem:convex-geodesics} below.
\end{remark}

\section{Gradient Comparison}\label{sec:gradient}

In this section, we prove the key gradient comparison between the first eigenfunction on $M$ and the first eigenfunction of the one-dimensional model. The comparison is the technical core of the paper, and its proof occupies the entirety of this section. The strategy is to assume, for contradiction, that the gradient of $u$ exceeds the model gradient at some point, and then to derive a contradiction by applying the maximum principle to a carefully chosen auxiliary function $Q$. The auxiliary function $J$ constructed in Lemma \ref{lem:J} is used to absorb the Einstein deviation $\Theta$, which would otherwise obstruct the maximum principle computation.

Figure~\ref{fig:gradient} illustrates the geometric content of the gradient comparison. The claim is that the gradient of the eigenfunction $u$ on $M$, measured with respect to the auxiliary metric $J|\nabla u|^2$, is dominated by the gradient of the one-dimensional model function $w$, composed with the inverse of $w$.

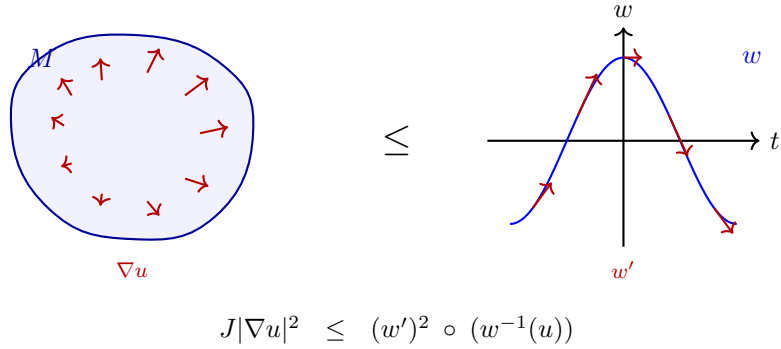
\begin{figure}[ht]
\centering
\begin{tikzpicture}[scale=1.0]
  \begin{scope}[shift={(-4.0,0)}]
    \draw[thick, blue!60!black, fill=blue!5] plot[smooth cycle, tension=0.8] coordinates {(-1.6,0.2) (-1.2,1.1) (0,1.4) (1.2,1.0) (1.6,0.1) (1.1,-1.0) (0,-1.3) (-1.1,-0.9)};
    \foreach \x/\y in {-0.8/0.6, -0.4/0.8, 0.2/0.9, 0.7/0.6, 0.9/0.1, 0.7/-0.5, 0.2/-0.8, -0.4/-0.7, -0.8/-0.3, -0.9/0.2} {
      \draw[->, red!70!black, thick] (\x,\y) -- ++(0.3*\x+0.1, 0.3*\y+0.05);
    }
    \node[blue!60!black, font=\small] at (-1.2,1.1) {$M$};
    \node[red!70!black, font=\scriptsize] at (0,-1.7) {$\nabla u$};
  \end{scope}
  \node[font=\Large] at (-0.5,0) {$\le$};
  \begin{scope}[shift={(2.5,0)}]
    \draw[->, thick] (-1.8,0) -- (1.8,0) node[right] {\small $t$};
    \draw[->, thick] (0,-1.4) -- (0,1.5) node[above] {\small $w$};
    \draw[thick, blue, domain=-1.5:1.5, samples=80] plot (\x, {1.1*cos(180*\x/1.5)});
    \node[blue, font=\small] at (1.7,1.1) {$w$};
    \foreach \x in {-1.2,-0.6,0,0.6,1.2} {
      \pgfmathsetmacro{\y}{1.1*cos(180*\x/1.5)}
      \pgfmathsetmacro{\dy}{-1.1*(180/1.5)*sin(180*\x/1.5)*(3.14159/180)}
      \draw[->, red!70!black, thick] (\x,\y) -- ++(0.25, 0.25*\dy);
    }
    \node[red!70!black, font=\scriptsize] at (0,-1.7) {$w'$};
  \end{scope}
  \node[font=\small, text width=8cm, align=center] at (-0.5,-2.5) {$J|\nabla u|^2 \le (w')^2 \circ (w^{-1}(u))$};
\end{tikzpicture}
\caption{The geometric content of the gradient comparison. The gradient of the eigenfunction $u$ on $M$, measured with respect to the auxiliary metric $J|\nabla u|^2$, is dominated by the gradient of the one-dimensional model function $w$, composed with the inverse of $w$. This comparison is the key to translating the one-dimensional estimate into a statement about $M$.}
\label{fig:gradient}
\end{figure}

\begin{theorem}\label{thm:gradient}
For every $\delta > 0$, there exist parameters $\tau > 1$, $\overline{n} > n$, and a constant $\varepsilon_0 = \varepsilon_0(n, p, D, \delta, R_E) > 0$ with the following property. If $\kappa_{\mathrm{Ein}}(p) < \varepsilon_0$, let $J$ be the positive solution of \eqref{eq:J} with $|J - 1| \le \delta$ and associated constant $\sigma$. Define $\overline{K}$ by
\[
\overline{K} = \begin{cases}
(1-\delta)K_E\frac{n-1}{\overline{n}-1} - (1+\delta)\frac{\sigma}{\overline{n}-1}, & K_E \ge 0, \\[6pt]
(1+\delta)K_E\frac{n-1}{\overline{n}-1} - (1+\delta)\frac{\sigma}{\overline{n}-1}, & K_E < 0,
\end{cases}
\]
where $K_E = R_E/(n(n-1))$ and $\sigma$ is the constant associated with $J$ as in Lemma \ref{lem:J}. Let $T$ be a solution of
\[
T' = \frac{T^2}{\overline{n}-1} + (\overline{n}-1)\overline{K},
\]
and let $w = w^{\overline{\lambda}}_{T,a}$ be the solution of the one-dimensional initial value problem \eqref{eq:ivp} on an interval $[a,b]$ such that $w' \ge 0$ on $[a,b]$, where
\[
\overline{\lambda} = (1+2\delta)\lambda_1.
\]
Assume that $[-1, u^*] \subset [-1, w(b)]$. Then
\begin{equation}\label{eq:gradient-comparison}
J|\nabla u|^2 \le (w')^2 \circ (w^{-1}(u)).
\end{equation}
\end{theorem}

\begin{proof}
We first consider the case $\partial M = \emptyset$. By contradiction, assume that
\[
J|\nabla u|^2 > (w'(w^{-1}(u(x))))^2
\]
at some point $x \in M$. For $c > 1$, define the auxiliary function
\[
Q = J|\nabla u|^2 - (cw')^2((cw)^{-1}(u(x))).
\]
Under the contradiction hypothesis, the function $c \mapsto Q_c(\overline{x}_c)$, where $\overline{x}_c$ denotes a maximum point of $Q_c$, satisfies $Q_1(\overline{x}_1) > 0$; moreover, $Q_c(\overline{x}_c) \to -\infty$ as $c \to \infty$. By continuity, there exists $c > 1$ such that $Q(\overline{x}) = 0$ at the maximal point $\overline{x}$, i.e.,
\[
J(\overline{x})|\nabla u|^2(\overline{x}) = (cw')^2((cw)^{-1}(u(\overline{x}))).
\]
At that point $\overline{x}$, we have $\nabla Q(\overline{x}) = 0$ and $\Delta Q(\overline{x}) \le 0$. In the sequel, we set $c = 1$; the general case $c > 1$ follows by the same argument after rescaling, and the choice $c = 1$ suffices for the desired contradiction.

\textbf{Step A: First-order conditions.} From $\nabla Q(\overline{x}) = 0$, we obtain
\[
\Hess u(\nabla u, \cdot) = c\frac{w''}{J}\nabla u - \frac{|\nabla u|^2}{2J}\nabla J,
\]
and hence
\[
A_u = c\frac{w''}{J} - \frac{1}{2J}\langle \nabla J, \nabla u\rangle,
\]
where $\langle \cdot, \cdot \rangle$ denotes the Riemannian inner product on tangent vectors. To verify the first identity, we compute $\nabla Q = (\nabla J)|\nabla u|^2 + 2J\,\Hess u(\nabla u,\cdot) - \psi'(u)\nabla u$, where $\psi(s) := (cw')^2((cw)^{-1}(s))$. Since $\psi'(s) = 2cw''(t)$ with $t = (cw)^{-1}(s)$, we obtain the stated identity.

\textbf{Step B: Second-order condition.} Using $\Delta Q(\overline{x}) \le 0$ together with the first-order conditions, we get
\begin{align}
0 &\ge \frac{1}{2}(\Delta J)|\nabla u|^2 + 2\frac{cw''}{J}\langle \nabla J, \nabla u\rangle - \frac{|\nabla J|^2}{J}|\nabla u|^2 + \frac{1}{2}J\Delta|\nabla u|^2 \nonumber \\
&\quad - c^2 w' w''' J^{-1} + c\lambda_1 w'' u. \label{eq:step1}
\end{align}
This follows from expanding $\Delta Q$ and using $|\nabla u|^2 = (cw')^2/J$ at the maximum point to rewrite the term $-\frac{w'''}{w'}|\nabla u|^2$ as $-c^2 w'w'''J^{-1}$.

\textbf{Step C: Bochner estimate.} From the decomposition
\[
\Ric = (n-1)K_E g + \theta g + \mathring{\Ric},
\]
we have
\[
\Ric(\nabla u, \nabla u) \ge \left[(n-1)K_E - \Theta\right]|\nabla u|^2,
\]
where $\Theta = |\mathring{\Ric}| + |\theta|$. Together with $\Delta u = -\lambda_1 u$, this gives
\begin{equation}\label{eq:bochner}
\frac{1}{2}\Delta|\nabla u|^2 \ge |\Hess u|^2 + \left[(n-1)K_E - \Theta - \lambda_1\right]|\nabla u|^2.
\end{equation}
For the Hessian term, we use the refined Cauchy--Schwarz inequality:
\begin{equation}\label{eq:hessian}
|\Hess u|^2 \ge \frac{\lambda_1^2 u^2}{n-1} + \frac{n}{n-1}A_u^2 + \frac{2\lambda_1 u}{n-1}A_u.
\end{equation}
The verification of \eqref{eq:hessian} proceeds by expanding $|\Hess u|^2 \ge \frac{(\Delta u)^2}{n} + \frac{n}{n-1}(A_u - \frac{\Delta u}{n})^2$ and substituting $\Delta u = -\lambda_1 u$.

\textbf{Step D: Reduction to a quadratic form.} Substituting \eqref{eq:bochner} and \eqref{eq:hessian} into \eqref{eq:step1} and applying the first-order condition, we obtain
\begin{align}
0 &\ge \frac{1}{2}\left(\Delta J - 2\frac{|\nabla J|^2}{J} - 2\Theta J\right)|\nabla u|^2 + \frac{(n-2)cw''}{(n-1)J}\langle \nabla u, \nabla J\rangle \nonumber \\
&\quad + \frac{n}{4J(n-1)}\langle \nabla J, \nabla u\rangle^2 + \frac{(n+1)\lambda_1 c^2 w'' w}{n-1} + \frac{J\lambda_1^2(cw)^2}{n-1} \nonumber \\
&\quad - \frac{\lambda_1 c w}{n-1}\langle \nabla J, \nabla u\rangle + \left((n-1)K_E - \lambda_1\right)(cw')^2 \nonumber \\
&\quad + c^2 J^{-1}\left(\frac{n}{n-1}(w'')^2 - w' w'''\right). \label{eq:step2}
\end{align}
We bound the mixed terms using Young's inequality with parameters $\alpha, \beta > 0$:
\begin{align}
-\frac{\lambda_1 c w}{n-1}\langle \nabla J, \nabla u\rangle &\ge -\alpha J\frac{\lambda_1^2 c^2 w^2}{n-1} - \frac{|\nabla J|^2|\nabla u|^2}{4\alpha(n-1)J}, \label{eq:young1} \\[4pt]
\frac{n-2}{n-1}\frac{cw''}{J}\langle \nabla J, \nabla u\rangle &\ge -\frac{n\beta}{n-1}\frac{c^2(w'')^2}{J} - \frac{(n-2)^2}{4\beta n(n-1)}\frac{|\nabla u|^2|\nabla J|^2}{J}. \label{eq:young2}
\end{align}
Applying \eqref{eq:young1} and \eqref{eq:young2} to \eqref{eq:step2} and using the equation satisfied by $J$,
\[
\frac{1}{2}\left(\Delta J - \left(2 + \frac{(n-2)^2}{2\beta n(n-1)} + \frac{1}{2(n-1)\alpha}\right)\frac{|\nabla J|^2}{J} - 2\Theta J\right)|\nabla u|^2 = -\sigma J|\nabla u|^2,
\]
together with the fact that $Q(\overline{x}) = 0$ implies $J|\nabla u|^2 = (cw')^2$, we arrive at
\begin{align}
0 &\ge J(1-\alpha)\frac{\lambda_1^2 c^2 w^2}{n-1} + \frac{n}{n-1}(1-\beta)\frac{c^2(w'')^2}{J} \nonumber \\
&\quad + \left((n-1)K_E - \lambda_1 - \sigma\right)(cw')^2 - c^2 w''' w' J^{-1} + \frac{n+1}{n-1}\lambda_1 c^2 w'' w. \label{eq:step3}
\end{align}
Using the one-dimensional model equation \eqref{eq:ivp}, we express $w'''w'$, $(w'')^2$, and $w''w$ in terms of $T$, $\overline{\lambda}$, $\overline{K}$, and $\overline{n}$. Specifically, from $w'' = Tw' - \overline{\lambda}w$ and the Riccati equation \eqref{eq:riccati}, we obtain
\[
w'''w' = \left(\frac{\overline{n}}{\overline{n}-1}T^2 + (\overline{n}-1)\overline{K} - \overline{\lambda}\right)(w')^2 - \overline{\lambda}Tw'w,
\]
\[
(w'')^2 = T^2(w')^2 - 2\overline{\lambda}Tw'w + \overline{\lambda}^2w^2, \qquad w''w = Tw'w - \overline{\lambda}w^2.
\]
Setting $y := J\lambda_1/\overline{\lambda}$, we obtain
\begin{equation}\label{eq:step4}
0 \ge \frac{1}{J(n-1)}Q_0(u, v) + \left((n-1)K_E - \lambda_1 - \sigma - (\overline{n}-1)\frac{\overline{K}}{J} + \frac{\overline{\lambda}}{J}\right)(w')^2,
\end{equation}
where we have renamed the quadratic form as $Q_0$ to avoid confusion with the maximum principle function $Q$, and where $u = \overline{\lambda}w$, $v = Tw'$, and
\[
Q_0(u, v) = P u^2 - S uv + R v^2,
\]
with
\[
P = (1-\alpha)y^2 - (n+1)y + n(1-\beta), \quad S = n+1 - 2n\beta - y(n+1), \quad R = \frac{\overline{n}-n}{\overline{n}-1} - n\beta.
\]
We note that the last term in \eqref{eq:step4} is obtained from the $(w')^2$ terms in \eqref{eq:step3} by factoring out $J^{-1}$ and using $J|\nabla u|^2 = (cw')^2$; this ensures that all terms in \eqref{eq:step4} have the same dimensional scaling. The sign of the $\overline{\lambda}/J$ term is positive, since it originates from the $-\overline{\lambda}(w')^2$ contribution in $w'''w'$ after multiplication by $-c^2 J^{-1}$.

\textbf{Step E: Positivity of $Q_0$ and of the remaining term.} Choose
\[
\alpha = \delta_1^2, \quad \beta = \frac{\delta_1^2}{n}, \quad \overline{n} = \frac{n - \eta_0\delta_1}{1 - \eta_0\delta_1}, \quad \eta_0 = \frac{(n+1)^2}{n-1}, \quad \delta_1 = \frac{\delta}{1+2\delta}.
\]
Then $P \ge 0$, $R \ge 0$, and $4PR \ge S^2$ for $\delta > 0$ sufficiently small. The positivity of $R$ follows from the explicit computation
\[
R = \eta_0\delta_1 - n\beta = \delta_1(\eta_0 - \delta_1) \ge 0
\]
for $\delta_1 < \eta_0$, where we used $\frac{\overline{n}-n}{\overline{n}-1} = \eta_0\delta_1$. The positivity of $P$ and of $4PR - S^2$ is verified as in \cite{RamosSetoTuerkoen} (Proof of Theorem 3.1).
With the choice of $\overline{K}$ as in the statement of the theorem, we have
\[
(n-1)K_E - (\overline{n}-1)\frac{\overline{K}}{J} - \sigma + \frac{\overline{\lambda}}{J} - \lambda_1 > 0
\]
for $\delta$ sufficiently small. To see this, we distinguish the cases $K_E \ge 0$ and $K_E < 0$. In the former case, the choice of $\overline{K}$ ensures that $(\overline{n}-1)\overline{K} \le (1-\delta)(n-1)K_E$, and hence the left-hand side is bounded below by $\delta(n-1)K_E - \sigma$, which is positive for $\delta$ sufficiently small and $\sigma \le 4\varepsilon$ with $\varepsilon$ sufficiently small. The latter case is analogous, with $(1+\delta)$ in place of $(1-\delta)$.
Combining these two positivity results yields a contradiction to \eqref{eq:step4}, completing the proof for $\partial M = \emptyset$. The case $\partial M \neq \emptyset$ follows from the convexity of the boundary, as in \cite{RamosSetoTuerkoen} (Lemma 3.4): at a boundary maximum point, the Neumann boundary condition ensures that the normal derivative of $Q$ vanishes, and the convexity of the boundary implies that the second fundamental form term has the correct sign.
\end{proof}

\begin{remark}
For $n \ge 3$, the maximum principle proof works without showing that the quadratic form is positive semi-definite. This simplification is possible because the additional dimensions provide enough room to absorb the mixed terms directly, without the need for the quadratic form analysis. We refer to \cite{RamosSetoTuerkoen} (Remark 3.3) for the details.
\end{remark}

\section{Maximum of Eigenfunctions}\label{sec:max}

In this section, we establish two results concerning the maximum of the first eigenfunction $u$. The first, Theorem \ref{thm:maxcomp}, asserts that the maximum $u^* = \max_M u$ of $u$ is bounded below by the maximum of a suitable one-dimensional model function. The second, Theorem \ref{thm:maxmatch}, asserts that the one-dimensional model can be chosen so that its maximum and minimum coincide with those of $u$. Together, these results provide the link between the eigenvalue problem on $M$ and the one-dimensional model problem, which is exploited in Section \ref{sec:proof} to complete the proof of Theorem \ref{thm:main}.

Figure~\ref{fig:maxmatch} illustrates the maximum matching theorem. The one-dimensional model function $w$ is chosen so that its maximum coincides with the maximum $u^*$ of the eigenfunction $u$, and its minimum coincides with the minimum $-1$ of $u$. This matching is achieved by continuously varying the starting point of the one-dimensional initial value problem.

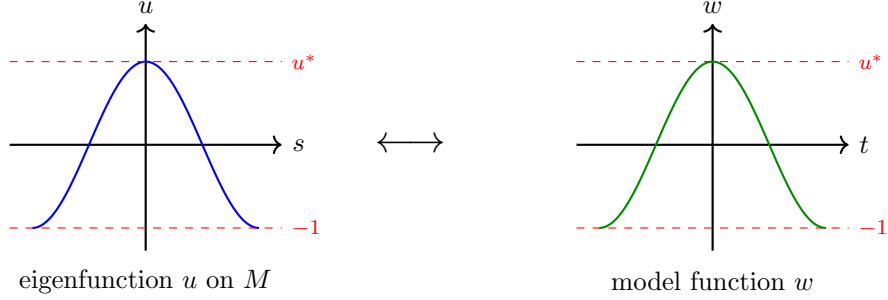
\begin{figure}[ht]
\centering
\begin{tikzpicture}[scale=1.0]
  \begin{scope}[shift={(-4.5,0)}]
    \draw[->, thick] (-1.8,0) -- (1.8,0) node[right] {\small $s$};
    \draw[->, thick] (0,-1.4) -- (0,1.6) node[above] {\small $u$};
    \draw[thick, blue, domain=-1.5:1.5, samples=80] plot (\x, {1.1*cos(180*\x/1.5)});
    \draw[dashed, red] (-1.8,1.1) -- (1.8,1.1);
    \draw[dashed, red] (-1.8,-1.1) -- (1.8,-1.1);
    \node[red, right, font=\scriptsize] at (1.8,1.1) {$u^*$};
    \node[red, right, font=\scriptsize] at (1.8,-1.1) {$-1$};
    \node[font=\small] at (0,-1.8) {eigenfunction $u$ on $M$};
  \end{scope}
  \node[font=\Large] at (-1.0,0) {$\longleftrightarrow$};
  \begin{scope}[shift={(3.0,0)}]
    \draw[->, thick] (-1.8,0) -- (1.8,0) node[right] {\small $t$};
    \draw[->, thick] (0,-1.4) -- (0,1.6) node[above] {\small $w$};
    \draw[thick, green!55!black, domain=-1.5:1.5, samples=80] plot (\x, {1.1*cos(180*\x/1.5)});
    \draw[dashed, red] (-1.8,1.1) -- (1.8,1.1);
    \draw[dashed, red] (-1.8,-1.1) -- (1.8,-1.1);
    \node[red, right, font=\scriptsize] at (1.8,1.1) {$u^*$};
    \node[red, right, font=\scriptsize] at (1.8,-1.1) {$-1$};
    \node[font=\small] at (0,-1.8) {model function $w$};
  \end{scope}
\end{tikzpicture}
\caption{The maximum matching theorem. The one-dimensional model function $w$ is chosen so that its maximum coincides with the maximum $u^* := \max_M u$ of the eigenfunction $u$, and its minimum coincides with the minimum $-1$ of $u$. This matching is achieved by continuously varying the starting point $a$ of the one-dimensional initial value problem \eqref{eq:ivp}.}
\label{fig:maxmatch}
\end{figure}

\subsection{Maximum comparison}

Let $\delta > 0$ be fixed, and let $\varepsilon > 0$, $\overline{n}$, $\overline{K}$, $\overline{\lambda}$ be as in Theorem \ref{thm:gradient}. In the case $\overline{K} > 0$, we choose it to be possibly slightly smaller but still positive:
\[
\overline{K} = \min\left\{\frac{nK_E(1 - C(n,p)\varepsilon)}{\overline{n}}, (1-\delta)K_E\frac{n-1}{\overline{n}-1} - (1+\delta)\frac{\sigma}{\overline{n}-1}\right\},
\]
where $C(n,p) > 0$ is the constant appearing in Aubry's estimate \cite{aubry2007finiteness}. The presence of the first term ensures that the model eigenvalue $\overline{\lambda}$ exceeds the threshold $\overline{n}\overline{K}$ required for the existence of the interval length $d_{\overline{n},\overline{K},\overline{\lambda}}$.

Define
\[
\overline{T}_{\overline{n},\overline{K}}(t) = \begin{cases}
(\overline{n}-1)\sqrt{\overline{K}}\tan(\sqrt{\overline{K}}t), & \overline{K} > 0, \\[2pt]
-(\overline{n}-1)\sqrt{-\overline{K}}\coth(\sqrt{-\overline{K}}t), & \overline{K} < 0.
\end{cases}
\]
We let $w = w^{\overline{\lambda}}_{\overline{T}_{\overline{n},\overline{K}},a}$ be the solution to the initial value problem \eqref{eq:ivp} with $a = 0$ if $\overline{K} < 0$ and $a = -\pi/(2\sqrt{\overline{K}})$ if $\overline{K} > 0$. We denote
\[
m_{\overline{n},\overline{K}} := w(a + d(\overline{T}_{\overline{n},\overline{K}}, a))
\]
the maximum value of $w$ on the interval $[a, a + d(\overline{T}_{\overline{n},\overline{K}}, a)]$.

\begin{theorem}\label{thm:maxcomp}
Let $p > n/2$ and $n \ge 2$. For every $\delta > 0$, there exists $\varepsilon > 0$ such that whenever $\kappa_{\mathrm{Ein}}(p) < \varepsilon$, we have
\[
u^* \ge m_{\overline{n},\overline{K}}.
\]
\end{theorem}

\begin{proof}[Proof of Theorem \ref{thm:maxcomp}, Part 1: Setup]
The proof follows the strategy of Bakry--Qian \cite{bakry2000some} and Naber--Valtorta \cite{naber2014sharp}, with modifications for the integral Einstein deviation. We assume for contradiction that $u^* < m_{\overline{n},\overline{K}}$.

Let $t_0 \in (a, b)$ denote the unique zero of $w$. Let $\varphi := w^{-1} \circ u$, so that $\varphi$ is a smooth function on $M$ with values in $[a, b]$; we use the symbol $\varphi$ rather than $g$ in order to avoid any conflict with the metric tensor $g$. Define the measure $m$ on $[a, b]$ by
\[
m(A) := \vol(\varphi^{-1}(A)), \quad A \subset [a, b] \text{ Borel},
\]
where $\vol(\cdot)$ denotes the Riemannian volume of a measurable subset of $M$. This implies that for any bounded and measurable function $f$ on $[a, b]$,
\[
\int_a^b f(s)\,dm(s) = \int_M f(\varphi(x))\,dV(x).
\]
The measure $m$ encodes the distribution of the values of $u$ relative to the model function $w$, and its properties are central to the argument.
\end{proof}

\begin{proposition}\label{prop:monotone}
Assume that $1 - \delta < J < 1 + \delta$, and $u$ and $w$ satisfy the conditions of Theorem \ref{thm:gradient}. Then the function
\[
E(s) := -\exp\left((1-\delta)\lambda_1\int_{t_0}^s \frac{w(t)}{w'(t)}dt\right)\int_a^s w(r)\,dm(r)
\]
is increasing on $(a, t_0]$ and decreasing on $[t_0, b)$.
\end{proposition}

\begin{proof}
Choose a smooth nonnegative function $H(s)$ with compact support in $(a, b)$. Define $G: [-1, w(b)] \to \mathbb{R}$ by
\[
\frac{d}{dt}[G(w(t))] = H(t), \quad G(-1) = 0.
\]
Choose $F$ satisfying $F(t) + tF'(t) = G(t)$. Then
\[
\Delta(uF(u)) = G(u)\Delta u + |\nabla u|^2 G'(u).
\]
Since $\int_M \Delta(uF(u))\,dV = 0$, we have
\[
\int_M (\Delta u)G(u)\,dV = -\int_M G'(u)|\nabla u|^2\,dV.
\]
Applying the gradient comparison \eqref{eq:gradient-comparison} and the eigenvalue equation, we get
\[
\lambda_1\int_M uG(u)\,dV \le (1-\delta)^{-1}\int_M G'(u)(w' \circ w^{-1})^2(u)\,dV.
\]
By the definition of $m$ and since $w(\varphi) = u$, we obtain
\[
\lambda_1\int_a^b w(s)G(w(s))\,dm(s) \le (1-\delta)^{-1}\int_a^b G'(w(s))(w'(s))^2\,dm(s).
\]
Since $G'(w(s))(w'(s))^2 = H(s)w'(s)$ and $G(w(s)) = \int_a^s H(r)\,dr$, we have
\[
\int_a^b \left(\int_s^b \lambda_1 w(r)m(dr)\right)H(s)\,ds \le \int_a^b (1-\delta)^{-1}H(s)w'(s)m(ds).
\]
Using $\int_M u\,dV = 0$, which follows from the fact that $u$ is orthogonal to the constant eigenfunction in the Neumann problem, we have $\int_a^b w\,dm = 0$, so $\int_s^b w\,dm = -\int_a^s w\,dm$. Let $A(s) := -\lambda_1\int_a^s w\,dm$. Then
\[
(1-\delta)^{-1}w'\,dm - A\,ds \ge 0.
\]
On $(a, t_0]$, since $w \le 0$ and $w' \ge 0$, this implies
\[
dA + (1-\delta)\lambda_1\frac{w}{w'}A\,ds \ge 0.
\]
Hence $E(s) = A(s)\exp\left((1-\delta)\lambda_1\int_{t_0}^s \frac{w}{w'}\,dr\right)$ is increasing on $(a, t_0]$ and decreasing on $[t_0, b)$.
\end{proof}

\begin{proposition}\label{prop:integral}
Let $q := (1-\delta)\lambda_1/\overline{\lambda}$. For $\varepsilon > 0$ sufficiently small, there exists a constant $C > 0$ independent of $\varepsilon$ such that
\[
\int_{\{u \le -1+\varepsilon\}} (-u)\,dV \le C\left(\int_{\{w \le -1+\varepsilon\}} (-w)\,d\nu\right)^q,
\]
where $d\nu(t) := \mu_{\overline{n},\overline{K}}(t)dt$ and $\mu_{\overline{n},\overline{K}}$ is the integrating factor satisfying $(\mu_{\overline{n},\overline{K}}w')' = -\overline{\lambda}\mu_{\overline{n},\overline{K}}w$.
\end{proposition}

\begin{proof}
Let $A(s) := \int_a^s (-w)\,dm = \int_{\{u \le w(s)\}}(-u)\,dV$ and $B(s) := \int_a^s (-w)\mu_{\overline{n},\overline{K}}\,dt$. From the eigenvalue equation $(\mu_{\overline{n},\overline{K}}w')' = -\overline{\lambda}\mu_{\overline{n},\overline{K}}w$, integrating from $a$ to $s$ gives
\[
\mu_{\overline{n},\overline{K}}(s)w'(s) = \overline{\lambda}B(s).
\]
Taking the logarithmic derivative and integrating from $s$ to $t_0$, we get
\[
\exp\left(\int_{t_0}^s \frac{w}{w'}\,dt\right) = \left(\frac{B(t_0)}{B(s)}\right)^{1/\overline{\lambda}}.
\]
By Proposition \ref{prop:monotone}, $E(s) \le E(t_0) = A(t_0)$ for $s \le t_0$. Substituting the expression for the exponential factor gives
\[
A(s)\left(\frac{B(t_0)}{B(s)}\right)^{(1-\delta)\lambda_1/\overline{\lambda}} \le A(t_0),
\]
which yields $A(s) \le CB(s)^q$. Evaluating at $s = w^{-1}(-1+\varepsilon)$ completes the proof.
\end{proof}

\begin{lemma}\label{lem:ball}
The preimage $u^{-1}([-1, -1+\varepsilon))$ contains a ball of radius
\[
r_\varepsilon = \sqrt{1-\delta}\left(w^{-1}(-1+\varepsilon) - a\right).
\]
\end{lemma}

\begin{proof}
Let $x_0$ be a minimum point of $u$, and let $\gamma$ be a unit speed minimizing geodesic from $x_0$. Set $f := u \circ \gamma$. By the gradient comparison,
\[
|f'| \le |\nabla u| \le \frac{w'(w^{-1}(f))}{\sqrt{J}} \le \frac{w'(w^{-1}(f))}{\sqrt{1-\delta}}.
\]
Hence $\frac{d}{dt}w^{-1}(f) \le 1/\sqrt{1-\delta}$, which implies
\[
a \le w^{-1}(f(t)) \le a + \frac{t}{\sqrt{1-\delta}}.
\]
Consequently, if $t < \sqrt{1-\delta}(w^{-1}(-1+\varepsilon) - a)$, then $u(\gamma(t)) < -1+\varepsilon$.
\end{proof}

\begin{proof}[Proof of Theorem \ref{thm:maxcomp}, Part 2: Conclusion]
By Proposition \ref{prop:integral} and Lemma \ref{lem:ball}, for $\tau \in (a, t_0)$ sufficiently close to $a$, with $h_\tau := w(\tau) = -1+\varepsilon_\tau$, we have
\[
\vol(\{u \le h_\tau\}) \le 2C'\nu([a, \tau])^q = 2C'\left(\int_a^\tau \mu_{\overline{n},\overline{K}}(t)\,dt\right)^q.
\]
Near $t = a$, the density $\mu_{\overline{n},\overline{K}}(t)$ behaves like $(t-a)^{\overline{n}-1}$, so the integral is bounded by $C''(\tau-a)^{\overline{n}}$. Hence
\[
\vol(B(x_0, r_\tau)) \le \tilde{C} r_\tau^{\overline{n}q}.
\]
We now show that $\overline{n}q > n$ for $\delta$ sufficiently small. With $\overline{\lambda} = (1+2\delta)\lambda_1$ and $\delta_1 = \delta/(1+2\delta)$, we have $\delta = \delta_1/(1-2\delta_1)$ and hence
\[
q = \frac{1-\delta}{1+2\delta} = 1 - 3\delta_1.
\]
Since $\overline{n} = (n - \eta_0\delta_1)/(1 - \eta_0\delta_1)$, we obtain
\[
\overline{n}q = \frac{(n - \eta_0\delta_1)(1 - 3\delta_1)}{1 - \eta_0\delta_1} = \frac{n - (\eta_0 + 3n)\delta_1 + 3\eta_0\delta_1^2}{1 - \eta_0\delta_1}.
\]
The condition $\overline{n}q > n$ is equivalent to
\[
n - (\eta_0 + 3n)\delta_1 + 3\eta_0\delta_1^2 > n(1 - \eta_0\delta_1) = n - n\eta_0\delta_1,
\]
which simplifies to
\[
\eta_0(n-1)\delta_1 - 3n\delta_1 + 3\eta_0\delta_1^2 > 0,
\]
or, after dividing by $\delta_1 > 0$,
\[
\eta_0(n-1) - 3n + 3\eta_0\delta_1 > 0.
\]
Since $\eta_0(n-1) = (n+1)^2$, we have $(n+1)^2 - 3n = n^2 - n + 1 \ge 3 > 0$ for $n \ge 2$. Hence $\overline{n}q > n$ holds strictly for all sufficiently small $\delta > 0$.

The inequality $\vol(B(x_0, r_\tau)) \le \tilde{C} r_\tau^{\overline{n}q}$ with $\overline{n}q > n$ contradicts the fact that $\vol(B(x_0, r)) \ge c r^n$ for small $r$, where $c > 0$ is a dimensional constant. Therefore $u^* \ge m_{\overline{n},\overline{K}}$.
\end{proof}

\subsection{Maximum matching}

\begin{theorem}\label{thm:maxmatch}
Given $\delta > 0$, there exist $\varepsilon > 0$, $\overline{K}$, and $\overline{n}$ (as in Theorem \ref{thm:gradient}) such that whenever $\kappa_{\mathrm{Ein}}(p) < \varepsilon$, there exist $T$ satisfying the Riccati equation \eqref{eq:riccati}, an interval $I = [\overline{a}, b(\overline{a})]$, and a corresponding Neumann eigenfunction $w = w^{\overline{\lambda}}_{T,\overline{a}}$ on $I$ such that
\[
u^* = \max_I w, \quad -1 = \min_M u = \min_I w.
\]
\end{theorem}

\begin{proof}
\sloppy
The proof follows \cite{RamosSetoTuerkoen} (Theorem 4.5) with $\rho_K$ replaced by $\Theta$. We divide into cases $\overline{K} > 0$ and $\overline{K} < 0$.

\textbf{Case 1: $\overline{K} > 0$.} By our choice of $\overline{K}$ and Aubry's estimate \cite{aubry2007finiteness}, we have
\[
\overline{\lambda} > \overline{n}\overline{K} = \lambda_1\left(T_{\overline{n},\overline{K}}, \left[-\frac{\pi}{2\sqrt{\overline{K}}}, \frac{\pi}{2\sqrt{\overline{K}}}\right]\right).
\]
If $u^* = 1$, we can find $-\pi/(2\sqrt{\overline{K}}) < \overline{a} < 0$ such that the Neumann eigenvalue of $\mathcal{L}_{T_{\overline{n},\overline{K}}}$ on $[\overline{a}, -\overline{a}]$ equals $\overline{\lambda}$, and the corresponding eigenfunction $w$ is odd with $\max w = 1$.

If $u^* < 1$, by Theorem \ref{thm:maxcomp}, $u^* \ge m_{\overline{n},\overline{K}}$. By the continuity of the maximum with respect to the starting point (see \cite{bakry2000some}, Section 3, Corollary 1), there exists $\overline{a}$ such that the corresponding Neumann eigenfunction on $[\overline{a}, b(\overline{a})]$ satisfies the matching condition.

\textbf{Case 2: $\overline{K} < 0$.} This follows from \cite{naber2014sharp} (Propositions 28 and 32) as in \cite{RamosSetoTuerkoen} (Theorem 4.5). The two subcases $\overline{\lambda} \le -(\overline{n}-1)^2\overline{K}/4$ and $\overline{\lambda} > -(\overline{n}-1)^2\overline{K}/4$ are treated separately, corresponding to whether the model potential is above or below the critical threshold.
\end{proof}

\section{Proof of Theorem \ref{thm:main}}\label{sec:proof}

In this section, we complete the proof of Theorem \ref{thm:main} by means of a diameter comparison. The argument combines the gradient comparison of Section \ref{sec:gradient} with the maximum matching of Section \ref{sec:max}, and concludes by comparing the eigenvalue of the manifold with that of the one-dimensional model.

Figure~\ref{fig:diameter} illustrates the diameter comparison argument. A minimizing geodesic $\gamma$ connects the minimum point $x_0$ and the maximum point $y_0$ of the eigenfunction $u$. Along this geodesic, the change of $u$ is controlled by the gradient comparison, and integrating over the interval $[-1, u^*]$ yields a lower bound for the diameter $D$ in terms of the length of the one-dimensional model interval.

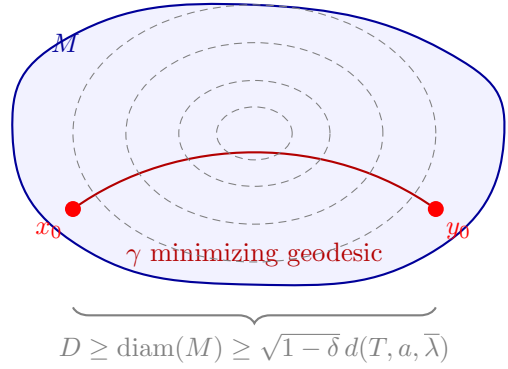
\begin{figure}[ht]
\centering
\begin{tikzpicture}[scale=1.0]
  \draw[thick, blue!60!black, fill=blue!5] plot[smooth cycle, tension=0.8] coordinates {(-3.2,0.2) (-2.4,1.5) (0,1.9) (2.5,1.3) (3.3,0.1) (2.3,-1.4) (0,-1.8) (-2.3,-1.3)};
  \node[blue!60!black, font=\small] at (-2.5,1.4) {$M$};
  \draw[thick, red!70!black] (-2.4,-0.8) .. controls (-1.0,0.2) and (1.0,0.2) .. (2.4,-0.8);
  \fill[red] (-2.4,-0.8) circle (3pt);
  \fill[red] (2.4,-0.8) circle (3pt);
  \node[red, below left, font=\small] at (-2.4,-0.85) {$x_0$};
  \node[red, below right, font=\small] at (2.4,-0.85) {$y_0$};
  \node[red!70!black, font=\small] at (0,-1.4) {$\gamma$ minimizing geodesic};
  \draw[gray, densely dashed] (0,0.2) ellipse (0.5 and 0.35);
  \draw[gray, densely dashed] (0,0.2) ellipse (1.0 and 0.7);
  \draw[gray, densely dashed] (0,0.2) ellipse (1.7 and 1.2);
  \draw[gray, densely dashed] (0,0.2) ellipse (2.4 and 1.7);
  \draw[decorate, decoration={brace, amplitude=6pt, mirror}, thick, gray] (-2.4,-2.1) -- (2.4,-2.1);
  \node[gray, below, font=\small] at (0,-2.3) {$D \ge \diam(M) \ge \sqrt{1-\delta}\,d(T,a,\overline{\lambda})$};
\end{tikzpicture}
\caption{The diameter comparison. A minimizing geodesic $\gamma$ connects the minimum point $x_0$ and the maximum point $y_0$ of the eigenfunction $u$. The level curves of $u$ are shown as dashed curves. Along $\gamma$, the gradient comparison controls the change of $u$, and integrating over $[-1, u^*]$ yields the lower bound $D \ge \sqrt{1-\delta}\,d(T,a,\overline{\lambda})$ in terms of the one-dimensional model. This is the final step in the proof of Theorem \ref{thm:main}.}
\label{fig:diameter}
\end{figure}

\begin{proof}[Proof of Theorem \ref{thm:main}]
Let $u$ denote the first nontrivial eigenfunction on $M$ and $\lambda_1$ the corresponding eigenvalue. We assume that $u$ is scaled such that $-1 = \min_M u < \max_M u = u^* \le 1$. By Theorem \ref{thm:maxmatch}, we can find an interval $[a, b]$, a function $T = T_{\overline{n},\overline{K}}$ such that the corresponding Neumann eigenfunction $w$ on $[a, b]$ satisfies
\[
\max_M u = \max_{[a,b(a)]} w.
\]
Now consider a normalized, minimizing geodesic $\gamma: [0, l] \to M$ connecting the minimum point $x_0$ and the maximum point $y_0$. Let $f(t) := u(\gamma(t))$ and choose $I \subset [0, l]$ such that $I \subset (f')^{-1}(0, \infty)$ and $f^{-1}$ is well-defined on a subset of full measure of $[-1, u^*]$. Then, by the change of variables formula and the gradient comparison,
\[
D \ge \diam(M) \ge \int_0^l dt \ge \int_I dt = \int_{-1}^{u^*} \frac{dy}{f'(f^{-1}(y))} \ge \int_{-1}^{u^*} \frac{\sqrt{1-\delta}\,dy}{w'(w^{-1}(y))} = \sqrt{1-\delta}\,d(T, a, \overline{\lambda}).
\]
Here the factor $\sqrt{1-\delta}$ arises from the gradient comparison \eqref{eq:gradient-comparison} and the lower bound $J \ge 1-\delta$; since $0 < \sqrt{1-\delta} < 1$, this yields a slightly weaker constant than in the pointwise curvature case, but the argument remains valid.
By Proposition \ref{prop:central}, $d(T, a, \overline{\lambda}) \ge d_{\overline{n},\overline{K},\overline{\lambda}}$. Hence $D \ge \sqrt{1-\delta}\,d_{\overline{n},\overline{K},\overline{\lambda}}$.
By the domain monotonicity \eqref{eq:domain-monotonicity}, $d_{\overline{n},\overline{K},\cdot}$ is decreasing in $\overline{\lambda}$. Hence there exists $C_1 = C_1(\delta) > 1$ with $C_1 \to 1$ as $\delta \to 0^+$ such that
\[
\sqrt{1-\delta}\,d_{\overline{n},\overline{K},\overline{\lambda}} = d_{\overline{n},\overline{K},C_1\overline{\lambda}}.
\]
Since $\overline{n} \to n$ and $\overline{K} \to K_E$ as $\delta \to 0^+$, there exists $C_2 = C_2(\delta) > 0$ with $C_2 \to 1$ such that
\[
d_{\overline{n},\overline{K},C_1\overline{\lambda}} = d_{n,K_E,C_1C_2\overline{\lambda}}.
\]
Since $d_{n,K_E,\cdot}$ is strictly decreasing and continuous, we conclude
\[
C_1C_2\overline{\lambda} \ge \lambda_1(n, K_E, D).
\]
In other words, since $\overline{\lambda} = (1+2\delta)\lambda_1$, for any $\alpha \in (0,1)$, there exists $\delta > 0$ such that
\[
\lambda_1 \ge \alpha \, \lambda_1(n, K_E, D) = \alpha \, \lambda_1\!\left(n, \frac{R_E}{n(n-1)}, D\right).
\]
Finally, note that if $\delta = 0$, i.e., if $\kappa_{\mathrm{Ein}}(p) = 0$, we recover the estimate $\lambda_1 \ge \lambda_1(n, K_E, D)$. This completes the proof.
\end{proof}

\begin{remark}
The proof shows that the constant $\varepsilon_0$ depends on $n$, $p$, $\alpha$, $D$, and $R_E$. The dependence on $R_E$ enters through the choice of $\overline{K}$ and through the Aubry estimate \cite{aubry2007finiteness}. In the case $R_E \le 0$, the argument requires minor modifications, since the model problem \eqref{eq:model} then has a different asymptotic behavior; we refer to \cite{RamosSetoTuerkoen} (Section 4) for the details.
\end{remark}

\appendix
\section{Global Neumann Sobolev Inequality under Integral Einstein Deviation Bounds}\label{app:sobolev}

In this appendix, we prove the global Neumann Sobolev inequality and Neumann Poincar\'e inequality under integral Einstein deviation bounds, which are needed for the Moser iteration in Lemma \ref{lem:J}. The argument follows \cite{RamosSetoTuerkoen} (Appendix A) with $\rho_K$ replaced by $\Theta$. The main point requiring verification is that the pointwise Jacobian comparison and volume comparison estimates, originally stated under integral Ricci curvature bounds, remain valid under integral Einstein deviation bounds. This is a consequence of the fact that the proofs of these estimates depend only on the $L^p$-norm of the curvature deficit, and the traceless Ricci tensor together with the scalar curvature deviation plays the same role as the Ricci curvature deficit in these arguments.

We use the averaged $L^q$ norm
\[
\|u\|_q^* := \left(\frac{1}{\vol(M)}\int_M |u|^q\,dv\right)^{1/q},
\]
defined for any $q > 0$ and any measurable function $u$ on $M$.

\begin{lemma}\label{lem:convex-geodesics}
Let $\intt(M)$ denote the interior of $M$. If $x, y \in \intt(M)$, then every minimizing geodesic joining $x$ and $y$ is contained in $\intt(M)$.
\end{lemma}

\begin{proof}
Since $M$ is compact, it is complete as a metric space, hence any two points can be joined by a minimizing geodesic for the intrinsic distance. Let $\gamma: [0, \ell] \to M$ be such a minimizing geodesic joining $x$ and $y$. Assume by contradiction that $\gamma(t_0) \in \partial M$ for some $t_0 \in (0, \ell)$. Let $t_0$ be the first such time. Since $\gamma([0, t_0)) \subset \intt(M)$ and $\gamma$ is minimizing, the tangent vector $\gamma'(t_0)$ must be tangent to $\partial M$ at $\gamma(t_0)$. Convexity of $\partial M$ then implies that one can shorten the portion of $\gamma$ near $t_0$ by replacing it with a nearby curve lying in the interior of $M$, contradicting the minimizing property of $\gamma$.
\end{proof}

\begin{proposition}\label{prop:sobolev}
Let $p > n/2$. There exists $\varepsilon = \varepsilon(n, p, R_E, D) > 0$ such that if $\kappa_{\mathrm{Ein}}(p) \le \varepsilon$, then there exists $C_S = C_S(n, p, R_E, D)$ such that for every $u \in W^{1,2}(M)$, where $W^{1,2}(M)$ denotes the Sobolev space of functions on $M$ whose first derivatives are square-integrable,
\[
\|u - u_M\|_{\frac{2p}{p-1}}^* \le C_S\|\nabla u\|_2^*, \quad u_M := \frac{1}{\vol(M)}\int_M u\,dV.
\]
Here $u_M$ denotes the average of $u$ over $M$. Consequently,
\[
\|u\|_{\frac{2p}{p-1}}^* \le C_S\|\nabla u\|_2^* + \|u\|_2^*.
\]
\end{proposition}

\begin{proof}
\sloppy
The proof follows \cite{RamosSetoTuerkoen} (Appendix A) verbatim. The only modification is that the pointwise Jacobian comparison and volume comparison estimates under integral Ricci curvature bounds (see \cite{petersen1997relative}) are replaced by their analogues under integral Einstein deviation bounds. Specifically, the relative volume comparison theorem of Petersen--Wei \cite{petersen1997relative} applies with $\rho_K$ replaced by $\Theta$, since the proof only uses the $L^p$ bound on the curvature deficit. The convexity of the boundary ensures that minimizing geodesics joining interior points remain in the interior (see Lemma \ref{lem:convex-geodesics}), and the Neumann boundary conditions eliminate boundary terms in the integration by parts. When $\partial M = \emptyset$, the argument simplifies, since the boundary terms are absent from the outset.
\end{proof}

\begin{corollary}\label{cor:poincare}
Let $p > n/2$. There exists $\varepsilon = \varepsilon(n, p, R_E, D) > 0$ such that if $\kappa_{\mathrm{Ein}}(p) \le \varepsilon$, then there exists $C_P = C_P(n, p, R_E, D) > 0$ such that for every $u \in W^{1,2}(M)$,
\[
\|u - u_M\|_2^* \le C_P\|\nabla u\|_2^*.
\]
Equivalently,
\[
\lambda_1(M) \ge C_P^{-2}.
\]
\end{corollary}

\begin{proof}
\sloppy
This follows from the Neumann isoperimetric inequality and Cheeger's inequality, as in \cite{RamosSetoTuerkoen} (Corollary A.12). The Neumann isoperimetric inequality is obtained by combining the volume comparison estimate, the local relative isoperimetric estimate, and the Vitali covering argument; the details are unchanged from \cite{RamosSetoTuerkoen}, with $\rho_K$ replaced by $\Theta$ throughout.
\end{proof}

\end{document}